\documentclass{article}

\usepackage{arxiv}

\usepackage[utf8]{inputenc} 
\usepackage[T1]{fontenc}    
\usepackage{hyperref}       
\usepackage{url}            
\usepackage{booktabs}       
\usepackage{amsfonts}       
\usepackage{nicefrac}       
\usepackage{microtype}      
\usepackage{lipsum}
\usepackage{graphicx}
\usepackage{amsmath}
\usepackage{amssymb}
\usepackage{bbold}
\graphicspath{ {./images/} }
\usepackage{epsfig,rotating,amsmath,subfigure}

\title{Technical comments on the optimal control for linear distributed time-delay}

\author{
 Jorge-Manuel Ortega-Mart\'inez \\
  La Salle Morelia University\\
  Morelia, Michoac\'an \\
  \texttt{jmortegam@ulsamorelia.edu.mx} \\
   \And
 Omar-Jacobo Santos-Sánchez \\
  CITIS-ICBI\\
  Autonomous Hidalgo State University\\
  Carretera Pachuca-Tulancingo, \\
  Mineral de la Reforma, Hidalgo, México\\
  \texttt{omarj@uaeh.edu.mx} \\
  \And
   Liliam Rodr\'iguez-Guerrero \\
  CITIS-ICBI\\
  Autonomous Hidalgo State University\\
  Carretera Pachuca-Tulancingo, \\
  Mineral de la Reforma, Hidalgo, México\\
  \texttt{rodriguez@uaeh.edu.mx} \\
  \And
 Sabine Mondi\'e \\
  Department of Automatic Control\\
  CINVESTAV-IPN\\
  Ciudad de México, M\'exico \\
  \texttt{smondie@ctrl.cinvestav.mx} \\
}
\newtheorem{theorem}{Theorem}
\newtheorem{proof}{Proof}
\newtheorem{definition}{Definition}
\newtheorem{remark}{Remark}
\newtheorem{proposition}{Proposition}
\newtheorem{corollary}{Corollary}

\begin{document}
\maketitle
\begin{abstract}
The solution to the infinite horizon optimal control problem for linear distributed time-delay systems is presented. The proposal is based on the use of the Cauchy solution for distributed time-delay systems. 
In contrast with previous results (for punctual time delay systems), the form of the functional and its properties are formally justified.
An important property is demonstrated: the Bellman functional admits upper and lower bounds, thus it is proved to be positive definite. Additionally, some experimental results on a temperature process are presented.
\end{abstract}


\section{Introduction}
Sufficient stability conditions for the infinite horizon optimal control problem for linear time-delay systems were obtained  using  Dynamic Programming by Krasovskii \cite{Krasovskii_1962} in the early sixties. This was achieved through a guessed three terms Bellman functional. Starting from this functional, Ross \cite{Ross_1969} characterized the optimal control law for the above mentioned problem. There, the author mentions briefly that the choice for the Bellman functional can be obtained from Riesz approximations; this was validated in the real-time optimal control of a dehydration process \cite{Lopez_2018}.  In \cite{Kushner_1970} the Cauchy formula allowed the construction of the functional of Bellman for the optimal control of linear distributed time-delay systems. However, as the finite horizon control problem was considered, stability in the sense of Lyapunov was not addressed. In \cite{Santos_2009}, \cite{Santos_2006}, Dynamic Programming combined with prescribed derivative functionals lead to an iterative procedure for the design of suboptimal control laws for systems with distributed delays. At each step, the obtained control law decreases the value of the performance index and simultaneously guarantees closed-loop stability. 

In \cite{kim2000linear}, the authors studied the exact solutions for the Riccati equations associated with the optimal control problem for linear distributed time-delay systems considering a quadratic performance index with state cross terms. These terms allowed to verify the positiveness of the proposed functional and to solve the Riccati equations.  However, the construction of the Bellman functional was omitted and the authors mentioned that "there is no effective criterion to verify the positive definiteness of the Bellman functional for the distributed time-delay systems" when a quadratic performance index without state cross terms is considered. Please notice that the sufficient conditions for optimality in the Dynamic Programming approach imply the existence of a positive definite functional (Bellman functional) which satisfies the Bellman equation.  Recently \cite{Ortega2021}, the construction of the Bellman functional for the optimal control problem for concentrated state-delay systems was addressed, but the formal proof of the Bellman functional positive definiteness was not approached.   

It appears that, to the best of the author's knowledge, the infinite horizon optimal control problem for linear distributed state delay systems, when a quadratic performance index without state cross terms, is not given in the literature. Moreover, the demonstration of the positive definiteness of the Bellman functional for linear state-delay systems is not reported. 
 
To address these problems, we construct in this paper the Bellman functional for distributed time-delay systems by combining the ideas in \cite{KharitonovZhabko_2003} and \cite{Ortega2021}, considering a quadratic performance index without state cross terms. We also prove additional properties of the Bellman functional: its existence and uniqueness, and, inspired by the result in \cite{Wenzhang_1989}, we present a cubic lower bound for the Bellman functional that connects optimality and stability, and allows the conclusion that the Bellman functional is positive definite. The conditions for the existence of a solution of the obtained Riccati equations are presented in \cite{Ross_1969}; it is associated with the existence of a solution of a certain hyperbolic quasi-linear partial differential equation.  Also, the relation between the Bellman functional with approximations of it, given in previous results \cite{Santos_2009}, is presented. Additionally, unprecedented trajectory tracking experiments for a temperature plant are presented.

The contribution is organized as follows: the preliminaries and problem statement are given in
section 2. The main results of this research are presented in section 3: there, the Bellman functional is constructed with the help of the Cauchy formula and is expressed in terms of delay Lyapunov matrices. Its properties are proved, as well as its existence and uniqueness. The Bellman functional upper and lower bounds are given in section 4. An illustrative example is given in section 5. Experimental results are reported in Section 6, and the contribution ends with some concluding remarks.

We denote the space of $\mathbb{R}^{n}$-valued piece wise-continuous functions on $[-h, 0]$ by $\mathrm{PC}([-h, 0],\mathbb{R}^{n})$. For a given initial function $\varphi(\theta)$,  $x_{t}(\varphi)$ denotes the state of the delay system $\{ x(t+\theta , \varphi), \theta \in [-h,0]\}$, with delay $h>0$; when the initial condition is not crucial, the argument $\varphi$ is omitted. The Euclidian norm for vectors is represented by $\parallel \cdot \parallel$. The set of piecewise continuous functions is equipped with the norm $\parallel \varphi \parallel_{h}=\sup_{\theta\in [-h,0]}\parallel \varphi(\theta) \parallel$. The notations $Q>0$ indicates that matrix $Q$ is positive definite. By $\dot{V}(x_{t})\mid_{\substack{ (*)  \\ u=u^{*}}}$, we denote the time derivative of the functional $V(x_{t})$ along the trajectories of system (*), when the control law is $u^{*}$.

\section{Preliminaries and Problem Statement}
Consider the linear distributed time-delay systems of the form
\begin{equation}
\begin{split}
&\dot{x}(t)=Ax(t)+Bx(t-h)+\int_{-h}^{0}E(\theta )x(t+\theta )d\theta+Du(t),\\
&\varphi \in \mathrm{PC}([-h, 0],\mathbb{R}^{n})
\end{split}
\label{sistema1}
\end{equation}
where the matrices $A,B \in \mathbb{R}^{n \times n}$, $D\in \mathbb{R}^{n \times r}$ are constant, $E(\theta)\in \mathbb{R}^{n \times n} $ is a continuous matrix defined for $\theta \in [-h,0]$, the state $x(t)$ is in $\mathbb{D}$, the space of solutions which contains the trivial one, and the control vector $u(t)$ belongs to $\mathbb{R}^{r}$, $r \leq n$.\\
Let the following quadratic performance index be given:
\begin{equation}
\begin{split}
J&=\int_{0}^{\infty} g(x_{t},u(t)) dt
=\int_{0}^{\infty} \left( x^{T}(t)Qx(t) + u^{T}(t)Ru(t)
\right) dt,
\end{split}
\label{desempenioJ}
\end{equation}
with $Q \in \mathbb{R}^{n \times n}, R \in \mathbb{R}^{r \times r}$, $Q > 0,R>0$.\\
The optimal control problem consists in the synthesis of the optimal control $u^{*}(t)$ that minimizes the quadratic performance index (\ref{desempenioJ}) subject to \eqref{sistema1}. 

Admissible controls for this problem satisfy:
\begin{enumerate}
\item $u=u(x_{t})$, in other words, the control is a function of the system state, $x_{t}.$
\item The functional $u(x_{t})$\ is such that solutions to (\ref{sistema1}) exist and are unique for $t\geq 0$\ and for all initial conditions $\varphi $.
\item The trivial solution of (\ref{sistema1}) in closed-loop with control law $u=u(x_{t})$\
is asymptotically stable.
\item For $u=u(x_{t})$ and all initial conditions $\varphi $ the performance
index has a finite value. 
\end{enumerate}
System (\ref{sistema1}) in closed-loop with an admissible control of the form
\begin{equation}
u_{L}(x_t)=\Gamma_{0}x(t)+\int_{-h}^{0} \Gamma_{1}(\theta)x(t+\theta) d\theta.
\label{u_gammas}
\end{equation}
is an exponentially stable system of the form
\begin{equation}
\dot{x}(t)= A_{0}x(t)+ A_{1}x(t-h)+\int_{-h}^{0}  G(\theta)x(t+\theta) d\theta, t\geq 0.
\label{sistema_1_lazo_cerrado_u_gammas}
\end{equation}
where  $x(t) \in \mathbb{D}$; $A_{0}=A+D\Gamma_{0},A_{1}=B,G(\theta)=E(\theta )+D\Gamma _{1}(\theta ) \in \mathbb{R}^{n \times n}$. As the right-hand side of  system (\ref{sistema_1_lazo_cerrado_u_gammas}) is Lipschitz, for any initial condition $\varphi\in \mathrm{PC}([-h, 0],\mathbb{R}^{n})$, the solution exists and is unique.  Sufficient conditions guaranteeing the existence of the control law (\ref{u_gammas}) can be found, for example, in \cite{rodriguez2019robust}. 

\begin{remark}
Notice that in Kim \cite{kim2000linear}, the optimal control (\ref{u_gammas}) was presented, the performance index (with state cross terms) to be minimized was:
\begin{equation}
    \begin{split}
    J_{K} =&\int_{0}^{\infty }\left( x^{T}\left( t\right) \Phi _{0}x\left(
    t\right) +2x^{T}\left( t\right) \int_{-h}^{0}\Phi _{1}\left( s\right)
    x\left( t+s\right) ds\right. \\
    &\left. +\int_{-h}^{0}\int_{-h}^{0}x^{T}\left( t+s\right) \Phi _{2}\left(
    s,v\right) x\left( t+v\right) dsdv+u^{T}\left( t\right) Nu\left( t\right) \right) dt,
    \end{split}
    \label{indice_kim}
\end{equation}
where $\Phi _{0}$ is a constant symmetric $n\times n$ matrix, $\Phi
_{1}\left( s\right) $ is an $n\times n$ matrix with piecewise-continuous
elements in $\left[ -h,0\right] $, $\Phi _{2}\left( s,v\right) $ is an $%
n\times n$ matrix with piecewise-continuous elements in $\left[ -h,0\right]
\times \left[ -h,0\right] $, and $N$ is an $r\times r$ positive-definite
symmetric matrix. The state weight functional in (\ref{indice_kim}) was proposed (without a construction process) by the quadratic functional
\begin{equation}
    \begin{split}
    V_{K}\left( x_{t}\right) =& x^{T}\left( t\right) \Phi _{0}x\left( t\right)
    +2x^{T}\left( t\right) \int_{-h}^{0}\Phi _{1}\left( s\right) x\left(t+s\right) ds
    +\int_{-h}^{0}\int_{-h}^{0}x^{T}\left( t+s\right) \Phi_{2}\left( s,v\right) x\left( t+v\right) dsdv.
    \end{split}
    \label{funcional_kim}
\end{equation}
According to \cite{kim2000linear}, the performance index \eqref{indice_kim} was used instead of \eqref{desempenioJ} because the number of free parameters increases, and it allowed the conclusion that the positiveness of the functional and to solve the Riccati equations.  \eqref{funcional_kim}.
\end{remark}
 The optimal control problem for punctual time-delay systems was studied  in \cite{Krasovskii_1962,Ross_1969}. The obtained results were built upon the following sufficient optimality conditions, showcasing the celebrated Bellman equation, generalized to time-delay systems.
\begin{theorem} \rm{(Ross \cite{Ross_1969}).}\label{teorema_2_1}
	\textit{If there exists an admissible control $u^{*}=u^{*}(x_{t})$ and a scalar continuous non negative functional $V(x_{t})$, $V=0$ for all $x_{t}=0$, such that
	\begin{equation}
	\dot{V}(x_{t})\mid_{\substack{ (\ref{sistema1})  \\ u=u^{*}}} +
	g(x_{t},u^{*}(x_{t}))=0, \quad \forall t\geq 0,
	\label{inciso_a}
	\end{equation}
	\begin{equation}
	\begin{split}
	&\dot{V}(x_{t}) \mid_{\substack{ (\ref{sistema1})  \\ u=u^{*}}} +
	g(x_{t},u^{*}(x_{t}))
	\leq \dot{V}(x_{t}) \mid_{\substack{ (\ref{sistema1})
			\\ u=u_{L}(x_{t})}} + g(x_{t},u), \quad \forall t\geq 0,
	\end{split}
	\label{inciso_b}
	\end{equation}	
	for all admissible $u_{L}(x_{t})$, then $u^{*}(x_{t})$ is an optimal control. Furthermore $V(\varphi)=J(\varphi,u^{*})$ is the optimal value of  the performance index $J$.}
\end{theorem}

The last result could be applied to the optimal control problem for distributed time-delay systems. The functional $V(x_{t})$ is called Bellman functional, notice that the requirement on the non negativity of the Bellman functional is instrumental in proving the necessary and sufficient conditions for an optimal control for time-delay systems. These conditions are given in the result below, which displays the analogue of the Riccati equation, for linear distributed time delay systems. 
\begin{theorem}\label{teorema_2_2}
	\textit{A linear control law
	\begin{equation}
	\begin{split}
	&u^{*}(t)=-R^{-1}D^{T} \Pi_{0} x(t)
	-R^{-1}D^{T}\int_{-h}^{0}
	\Pi_{1}(\theta)x(t+ \theta) d\theta, \quad t\geq 0,
	\end{split}
	\label{ley_control_u0}
	\end{equation}
	provides the global minimum of the performance index \eqref{desempenioJ} for the dynamical system \eqref{sistema1} if:}
	\begin{itemize}
	    \item[a)] \textit{$u^{*}(x_{t})$ is a stabilizing control law (since $u^{*}$ is linear, stability and admissibility are equivalent).}
		\item[b)] \textit{$\Pi_{0}$ is a symmetric positive definite matrix which, together with the $n \times n$ array $\Pi_{1}(\theta)$ of functions defined on $[-h,0]$, and the $n\times n$ array, $\Pi_{2}(\xi,\theta)$ of functions in two variables having domain $\xi, \theta \in [-h,0]$, satisfies the relations:}
	\end{itemize}
	\begin{itemize}
		\item[1)] $A^{T}\Pi _{0}+\Pi _{0}A-\Pi _{0}DR^{-1}D^{T}\Pi _{0}+\Pi _{1}^{T}\left(0\right) +\Pi _{1}\left( 0\right) +Q=0,$
		\item[2)] $\frac{d\Pi _{1}\left( \theta \right) }{d\theta }=\left( A^{T}-\Pi_{0}DR^{-1}D^{T}\right) \Pi _{1}\left( \theta \right) +\Pi _{2}\left(0,\theta \right) +\Pi _{0}E\left( \theta \right) ,\text{ \ }\theta \in \left[-h,0\right],$
		\item[3)] $\frac{\partial \Pi _{2}\left( \xi ,\theta \right) }{\partial \xi }+\frac{\partial \Pi _{2}\left( \xi ,\theta \right) }{\partial \theta }=-\Pi_{1}^{T}\left( \xi \right) DR^{-1}D^{T}\Pi _{1}\left( \theta \right)+2E^{T}\left( \xi \right) \Pi _{1}\left( \theta \right) , \quad \xi,\theta \in \left[ -h,0\right],$
		\item[4)] $\Pi _{1}\left( -h\right) =\Pi _{0}B,$
		\item[5)] $\Pi _{2}\left( -h,\theta \right) =B^{T}\Pi _{1}\left( \theta \right) ,\text{
\ }\theta \in \left[ -h,0\right].$
		\begin{equation}
		\label{SetPartialEquations}
		\end{equation}
	\end{itemize}
	\textit{Furthermore, the representation of \eqref{desempenioJ} in
	terms of the initial function is}
	\begin{equation}
	\begin{split}
	J(\varphi,u^{*})=&\varphi^{T}(0) \Pi_{0} \varphi(0) + 2 \varphi^{T}(0)
	\int_{-h}^{0} \Pi_{1}(\theta) \varphi (\theta) d \theta
	+ \int_{-h}^{0}	\int_{-h}^{0} \varphi^{T}(\xi) \Pi_{2} (\xi,\theta) \varphi(\theta) d \xi d
	\theta.
	\label{indice_desempenio_J_bajo_control_u0}
	\end{split}
	\end{equation} 
\end{theorem}
The proof of the above result is given later. It consists roughly to propose an admissible control candidate for the optimal control, the construction of the Bellman functional associated with this candidate optimal control, after replacing the derivative of the Bellman functional and solving \eqref{inciso_b}. Finally, the Riccati equations are obtained with the optimal control found through the minimization process.  In the paper by Ross \cite{Ross_1969}, the form of the Bellman functional, for punctual time-delay systems, is introduced in his Proposition 3. It is followed by a brief discussion that outlines that the essence of the proposal is the use of Riesz approximations. However, we consider that the construction step of the Bellman functional and to prove its positiveness is a crucial stage. Here, a constructive proof for the Bellman functional is given. The following proposition (the proof is given later) establishes the form of the Bellman functional for the optimal control problem for distributed time-delay systems.

\begin{proposition} \label{proposicion_2_1}
	If $u_{L}=u_{L}(x_{t})$, $\forall t\geq 0$, is an admissible linear control for the system given by \eqref{sistema1}, $\varphi$ is an
	initial condition functional on $[-h,0]$,  then the function
	\begin{equation}\label{Vpositiva}
	V(\varphi) = J(\varphi, u_{L})= \int_{0}^{\infty}
	(x^{T}(t)Qx(t)+u_{L}(t)^{T}Ru_{L}(t)) dt,
	\end{equation}%
	can be expressed as
	\begin{equation}
	\begin{split}
	&V(\varphi) = \varphi^{T}(0) \Pi_{0} \varphi(0) +2 \varphi^{T}(0) \int_{-h}^{0}
	\Pi_{1}(\theta) \varphi (\theta) d \theta
	+ \int_{-h}^{0} \int_{-h}^{0} \varphi^{T}(\xi) \Pi_{2} (\xi,\theta)
	\varphi(\theta) d \xi d \theta,
	\end{split}
	\label{funcional_v_phi_11}
	\end{equation}
	where
	\begin{itemize}
		\item[i)] $\Pi_{0}>0$ is a symmetric positive matrix.
		\item[ii)] $\Pi_{1}(\theta)$ is defined on $[-h,0]$.	
		\item[iii)] $\Pi_{2}(\xi,\theta)$ is defined on $%
		\xi, \theta \in [-h,0], \quad \Pi_{2}^{T}(\xi, \theta)=\Pi_{2}(\theta,\xi)$.
	\end{itemize}
\end{proposition}

\textbf{Problem Statement.}
 The main purpose of this contribution is to solve the infinite horizon optimal control problem for distributed time-delay linear systems. To solve this problem, we construct the Bellman functional given by 	\eqref{funcional_v_phi_11}, with matrix $\Pi_{0}$, and matrix functions $\Pi_{1}(\cdot)$,  $\Pi_{2}(\cdot,\cdot)$ satisfying conditions \textit{i)}, \textit{ii)}, and \textit{iii)}, respectively.  We also prove a very important property of the Bellman functional: that it is positive definite.
 
 The key element for addressing this challenge is the delay Lyapunov matrix function for stable linear systems of the form \eqref{sistema_1_lazo_cerrado_u_gammas} introduced in \cite{Santos_2009}.  It is the analogue of the classical Lyapunov matrix for linear delay free systems, and is similar to the delay Lyapunov matrix for pointwise time-delay system presented in \cite{Kharitonov_2006}:

\begin{definition}\rm{(Santos et al. \cite{Santos_2009}).}\label{def_matris_lyapunov}
\textit{Let system \eqref{sistema_1_lazo_cerrado_u_gammas} be exponentially stable. Given a matrix $M$, the Lyapunov matrix function is defined as
\begin{equation}
U(\tau,M) = \int_{0}^{\infty} K^{T}(t)MK(t+\tau)dt,
\label{matriz_U_0_tau_M}
\end{equation}
where $K(t)$, the fundamental matrix of system \eqref{sistema_1_lazo_cerrado_u_gammas}, is such that
\begin{equation*}
\dot{K}(t)= K(t)A_{0}+K(t-h)A_{1}+\int^{0}_{-h}K(t+\theta) G(\theta) d\theta, \quad t\geq 0,
\end{equation*}
with initial condition $K(0)=I$, and $K(\theta)=0$ for $\theta \in [-h,0]$.}
\end{definition}
The following properties of matrix \eqref{matriz_U_0_tau_M} hold:
\begin{theorem} \rm{(Santos et al. \cite{Santos_2009}).}\label{teorema_santos}
	\textit{The Lyapunov matrix \eqref{matriz_U_0_tau_M} satisfies the conditions
		\begin{equation}
		\begin{split}
		\frac{d}{d\tau}U(\tau,M)=&U(\tau,M)A_{0}+U(\tau-h,M)A_{1} 
		+ \int_{-h}^{0}U(\tau+\theta,M)G(\theta+\tau)d\theta, \quad \tau \geq 0,
		\end{split}
		\label{dinamica_U}
		\end{equation}
		\begin{equation}
		U(-\tau,M) = U^{T}(\tau,M^{T}),
		\label{simetrica_U}
		\end{equation}
		\begin{equation}
		-M= U'(+0,M) - U'(-0,M).
		\label{algebraica_U}
		\end{equation}
	}
\end{theorem}

A delay Lyapunov matrix property of major significance in the development of our main results  is the following.
\begin{theorem} \rm{(Kharitonov \cite{Kharitonov_2006}).}\label{exisuni}
\textit{Let system \eqref{sistema_1_lazo_cerrado_u_gammas} be exponentially stable, then matrix
\eqref{matriz_U_0_tau_M} is the unique solution of matrix equations \eqref{dinamica_U}-\eqref{algebraica_U}}.
\end{theorem}

\section{Main results}
In this section we give a proof for the representation of the quadratic performance index when an admissible state feedback control is considered. When the control is optimal this representation coincides with the Bellman functional. Our proof is based on the Cauchy formula, which expresses the closed-loop solution in terms of the fundamental matrix. First, we provide an expression in terms of the fundamental matrix, and next, in terms of the delay Lyapunov matrix.

\begin{proof}[Proof of Proposition \ref{proposicion_2_1}]
Given system (\ref{sistema1}), consider an admissible feedback state control of the form given by (\ref{u_gammas}) such that the closed-loop system (\ref{sistema_1_lazo_cerrado_u_gammas}), is exponentially stable. The Cauchy formula \cite{Kolmanovskii_1992} for the distributed time-delay system \eqref{sistema_1_lazo_cerrado_u_gammas} is  :
\begin{equation}
\begin{split}
x(t,\varphi)=&K(t) \varphi(0) + \int_{-h}^{0} K(t-\theta-h)A_{1}\varphi(\theta) d
\theta
+ \int_{-h}^{0} \int_{-h}^{\theta} K(t- \theta+ \xi)G(\xi) d\xi
\varphi(\theta) d\theta, \quad t\geq 0
\end{split}
\label{solucion_cauchy}
\end{equation}
where $K(t)$ is the fundamental matrix of system \eqref{sistema_1_lazo_cerrado_u_gammas}. If the closed-loop system has exponentially stable trivial solution, then the fundamental matrix satisfies the inequality
\begin{equation}
|| K(t) || \leq \gamma e^{-\beta t}, \forall t \geq 0.
\label{Ktacotada}
\end{equation}

By defining
\begin{equation}
\hat{K}(t,\theta)=K(t-\theta-h)A_{1}+\int_{-h}^{\theta} K(t- \theta+ \xi)G(\xi)
d\xi,
\label{k_gorro}
\end{equation}
equation \eqref{solucion_cauchy} is rewritten as:
\begin{equation}
x(t,\varphi)= K(t)\varphi(0)+\int_{-h}^{0} \hat{K}(t,\theta) \varphi
(\theta) d\theta.
\label{sol_cauchy_reducida}
\end{equation}

Expressing the admissible control law \eqref{u_gammas} in terms of the Cauchy solution \eqref{sol_cauchy_reducida}, and substituting \eqref{sol_cauchy_reducida} into \eqref{desempenioJ}, we obtain the following representation of the performance index:
\begin{equation}
\begin{split}
J(\varphi ,u_{L}(x_{t})) =&\varphi ^{T}(0)\Pi _{0}\varphi
(0)+2\varphi(0) ^{T} \int_{-h}^{0}\Pi _{1}(\theta )\varphi(\theta )d\theta
+\int_{-h}^{0}\int_{-h}^{0}\varphi ^{T}(\xi )\Pi_{2}(\xi ,\theta
)\varphi (\theta )d\theta d\xi ,
\end{split}
\label{J_prueba_teorema}
\end{equation}%
with matrices $\Pi _{0}$, $\Pi _{1}(\theta )$ and $\Pi_{2}(\xi,\theta)$ defined as:
\begin{equation}
\begin{split}
\Pi_{0} =&\int_{0}^{\infty } K^{T}(t)M_{1}K(t)
dt
+2\int_{0}^{\infty }\int_{-h}^{0} K^{T}(t)M_{2}(\theta )K(t+\theta
) d\theta dt\\  
&+\int_{0}^{\infty }\int_{-h}^{0}\int_{-h}^{0}K^{T}(t+\theta
_{1})M_{3}(\theta _{1},\theta _{2})K(t+\theta _{2}) d\theta
_{1}d\theta _{2}dt,
\end{split}
\label{pi0barra}
\end{equation}	
\begin{equation}
\begin{split}
\Pi _{1}(\theta ) =&\int_{0}^{\infty }K^{T}(t)M_{1}\hat{K}(t,\theta
)dt
+2\int_{-h}^{0}\int_{0}^{\infty }K^{T}(t)\left(
M_{2}\left( \theta _{2}\right) \hat{K}\left( t+\theta _{2},\theta \right)
\right. +\left. M_{2}^{T}\left( \theta _{2}\right) \hat{K}\left( t-\theta
_{2},\theta \right) \right) dtd\theta _{2} \\
& +\int_{-h}^{0}\int_{-h}^{0}\int_{0}^{\infty }K^{T}(t+\theta
_{1})M_{3}(\theta _{1},\theta _{2})\hat{K}(t+\theta _{2},\theta )dtd\theta_{1}d\theta _{2}, 
\end{split}
\label{pi1barra}
\end{equation}
\begin{equation}
\begin{split}
\Pi_{2}(\xi ,\theta ) =&\int_{0}^{\infty }\hat{K}^{T}(t,\xi )M_{1}\hat{K}%
(t,\theta )dt
+2\int_{-h}^{0}\int_{0}^{\infty }\hat{K}^{T}(t,\xi
)M_{2}(\theta _{2})\hat{K}(t+\theta _{2},\theta )dtd\theta _{2}\\
& +\int_{-h}^{0}\int_{-h}^{0}\int_{0}^{\infty }\hat{K}^{T}(t+\theta
_{1},\xi )M_{3}(\theta _{1},\theta _{2})\hat{K}(t+\theta _{2},\theta
)dtd\theta _{1}d\theta _{2},
\end{split}
\label{pi2barra}
\end{equation}	
where
\begin{itemize}
	\setlength{\parskip}{0pt}
	\setlength{\itemsep}{0pt plus 1pt}
	\item $M_{1}= Q + \Gamma_{0}^{T}R\Gamma_{0}$,
	\item$M_{2}(\theta)=\Gamma_{0}^{T}R\Gamma_{1}(\theta)$, $\theta \in [-h,0]$
	\item$M_{3}(\theta_{1},\theta_{2})=\Gamma_{1}^{T}(\theta_{1})R\Gamma_{1}(\theta_{2})$, $\theta_{1}, \theta_{2} \in [-h,0]$.
\end{itemize}

As the closed-loop system is assumed to be exponentially stable, the matrices $\Pi_{0}$, $\Pi_{1}(\cdot)$ and $\Pi_{2}(\cdot,\cdot)$ are well defined because the fundamental matrix $K(t)$ satisfies \eqref{Ktacotada}, hence the convergence under the improper integrals is insured.

i) By substituting the matrices $M_1, M_2(\cdot)$ and $M_3(\cdot,\cdot)$, one can rewrite  (\ref{pi0barra}) as
\begin{eqnarray}
&&\Pi_{0}%
\begin{array}{c}
=%
\end{array}%
\int_{0}^{\infty } \Psi^{T}(t) \left[
\begin{array}{cc}
Q+\Gamma _{0}^{T}R\Gamma _{0}\text{ \ \ } & \Gamma _{0}^{T}R \\
R\Gamma _{0} & R%
\end{array}%
\right] \Psi(t) dt.  \label{simetrica_Pi_0}
\end{eqnarray}%
where:
$ \Psi^{T}(t)=\left[
\begin{array}{cc}
K^{T}(t) & \int_{-h}^{0}K^{T}(t+\theta )\Gamma _{1}^{T}(\theta )d\theta%
\end{array}%
\right].$
As the symmetric matrices $Q$ and $R$ in  (\ref{desempenioJ}) are positive definite, Schur complements imply that the right-hand side in (\ref{simetrica_Pi_0}) is positive definite.

\textit{ii)} The fact that the matrix $\Pi_{1}(\theta)$ is well defined on $[-h,0]$, follows directly from the properties of matrix $K(t)$ on the same interval. For the same reason, the matrix $\Pi_{1}(\theta)$ is also continuous over $[-h,0]$. 

\textit{iii)} As in ii), the matrix $\Pi_{2}(\xi,\theta)$ is well defined and continuous over $[-h,0]\times[-h,0]$. The proof of the symmetry property of matrix $\Pi_2(\xi,\theta)$ is somehow involved, but straightforward. It is obtained by direct transposition of the expression (\ref{pi2barra}), appropriate  changes of variables, along with the use of fundamental matrix properties, it shows in next expressions 
\begin{equation*}
\begin{split}
&\Pi _{2}^{T}(\xi ,\theta )= \int_{0}^{\infty }\bigg(A_{1}^{T}K^{T}(t-\theta
-h)M_{1}K(t-\xi -h)A_{1} 
+\int_{-h}^{\xi }A_{1}^{T}K^{T}(t-\theta -h)M_{1}K(t-\xi +\delta )G(\delta
)d\delta  \\
& +\int_{-h}^{\theta }G^{T}(\delta )K^{T}(t-\theta +\delta )M_{1}K(t-\xi
-h)A_{1}d\delta  
 +\int_{-h}^{\theta }\int_{-h}^{\xi }G^{T}(\delta _{1})K^{T}(t-\theta
+\delta _{1})M_{1}K(t-\xi +\delta _{2})G(\delta _{2})d\delta _{2}d\delta _{1}
\\
& +\int_{-h}^{0}A_{1}^{T}K^{T}(t-\theta -h)M_{2}(\theta _{2})K(t+\theta
_{2}-\xi -h)A_{1}d\theta _{2} 
+\int_{-h}^{0}A_{1}^{T}K^{T}(t+\theta _{2}-\theta -h)M_{2}^{T}(\theta
_{2})K(t-\xi -h)A_{1}d\theta _{2} \\
& +\int_{-h}^{0}\int_{-h}^{\xi }A_{1}^{T}K^{T}(t-\theta -h)M_{2}(\theta
_{2})K(t+\theta _{2}-\xi +\delta )G(\delta )d\delta d\theta _{2} \\
& +\int_{-h}^{0}\int_{-h}^{\xi }A_{1}^{T}K^{T}(t+\theta _{2}-\theta
-h)M_{2}^{T}(\theta _{2})K(t-\xi +\delta )G(\delta )d\delta d\theta _{2} \\
& +\int_{-h}^{0}\int_{-h}^{\xi }\int_{-h}^{\theta }G^{T}(\delta
_{2})K^{T}(t+\theta _{2}-\theta +\delta _{2})M_{2}^{T}(\theta _{2})K(t-\xi
+\delta _{1})G(\delta _{1})d\delta _{2}d\delta _{1}d\theta _{2} \\
& +\int_{-h}^{0}\int_{-h}^{\theta }G^{T}(\delta )K^{T}(t+\theta _{2}-\theta
+\delta )M_{2}^{T}(\theta _{2})K(t-\xi -h)A_{1}d\delta d\theta _{2}\\
\end{split}%
\end{equation*}%
\begin{equation}
\begin{split}
& +\int_{-h}^{0}\int_{-h}^{\theta }G^{T}(\delta )K^{T}(t-\theta +\delta
)M_{2}(\theta _{2})K(t+\theta _{2}-\xi -h)A_{1}d\delta d\theta _{2} \\
& +\int_{-h}^{0}\int_{-h}^{\theta }\int_{-h}^{\xi }G^{T}(\delta
_{1})K^{T}(t-\theta +\delta _{1})M_{2}(\theta _{2})K(t+\theta _{2}-\xi
+\delta _{2})G(\delta _{2})d\delta _{2}d\delta _{1}d\theta _{2} \\
& +\int_{-h}^{0}\int_{-h}^{0}A_{1}^{T}K^{T}(t+\theta _{1}-\theta
-h)M_{3}(\theta _{1},\theta _{2})K(t+\theta _{2}-\xi -h)A_{1}d\theta
_{1}d\theta _{2} \\
& +\int_{-h}^{0}\int_{-h}^{0}\int_{-h}^{\xi }A_{1}^{T}K^{T}(t+\theta
_{1}-\theta -h)M_{3}(\theta _{1},\theta _{2})K(t+\theta _{2}-\xi +\delta
)G(\delta )d\delta d\theta _{1}d\theta _{2} \\
& +\int_{-h}^{0}\int_{-h}^{0}\int_{-h}^{\theta }G^{T}(\delta )K^{T}(t+\theta
_{1}-\theta +\delta )M_{3}(\theta _{1},\theta _{2})K(t+\theta _{2}-\xi
-h)A_{1}d\delta d\theta _{1}d\theta _{2} \\
& +\int_{-h}^{0}\int_{-h}^{0}\int_{-h}^{\theta }\int_{-h}^{\xi }G^{T}(\delta
_{1})K^{T}(t+\theta _{1}-\theta +\delta _{1})M_{3}(\theta _{1},\theta
_{2})K(t+\theta _{2}-\xi +\delta _{2}) 
G(\delta _{2})d\delta _{2}d\delta _{1}d\theta _{1}d\theta _{2}%
\bigg)dt,
\end{split}
\label{Pi_2_Trans}
\end{equation}%
and
\begin{equation}
\begin{split}
&\Pi _{2}(\theta ,\xi )= \int_{0}^{\infty }\bigg(A_{1}^{T}K^{T}(t-\theta
-h)M_{1}K(t-\xi -h)A_{1} 
+\int_{-h}^{\xi }A_{1}^{T}K^{T}(t-\theta -h)M_{1}K(t-\xi +\delta )G(\delta
)d\delta  \\
& +\int_{-h}^{\theta }G^{T}(\delta )K^{T}(t-\theta +\delta )M_{1}K(t-\xi
-h)A_{1}d\delta  
+\int_{-h}^{\theta }\int_{-h}^{\xi }G^{T}(\delta _{1})K^{T}(t-\theta
+\delta _{1})M_{1}K(t-\xi +\delta _{2})G(\delta _{2})d\delta _{2}d\delta _{1}
\\
& +\int_{-h}^{0}A_{1}^{T}K^{T}(t-\theta -h)M_{2}(\theta _{2})K(t+\theta
_{2}-\xi -h)A_{1}d\theta _{2} 
+\int_{-h}^{0}A_{1}^{T}K^{T}(t+\theta _{2}-\theta -h)M_{2}^{T}(\theta
_{2})K(t-\xi -h)A_{1}d\theta _{2} \\
& +\int_{-h}^{0}\int_{-h}^{\xi }A_{1}^{T}K^{T}(t-\theta -h)M_{2}(\theta
_{2})K(t+\theta _{2}-\xi +\delta )G(\delta )d\delta d\theta _{2} \\
& +\int_{-h}^{0}\int_{-h}^{\xi }A_{1}^{T}K^{T}(t+\theta _{2}-\theta
-h)M_{2}^{T}(\theta _{2})K(t-\xi +\delta )G(\delta )d\delta d\theta _{2} \\
& +\int_{-h}^{0}\int_{-h}^{\xi }\int_{-h}^{\theta }G^{T}(\delta
_{2})K^{T}(t+\theta _{2}-\theta +\delta _{2})M_{2}^{T}(\theta _{2})K(t-\xi
+\delta _{1})G(\delta _{1})d\delta _{2}d\delta _{1}d\theta _{2} \\
& +\int_{-h}^{0}\int_{-h}^{\theta }G^{T}(\delta )K^{T}(t+\theta _{2}-\theta
+\delta )M_{2}^{T}(\theta _{2})K(t-\xi -h)A_{1}d\delta d\theta _{2} \\
& +\int_{-h}^{0}\int_{-h}^{\theta }G^{T}(\delta )K^{T}(t-\theta +\delta
)M_{2}(\theta _{2})K(t+\theta _{2}-\xi -h)A_{1}d\delta d\theta _{2} \\
& +\int_{-h}^{0}\int_{-h}^{\theta }\int_{-h}^{\xi }G^{T}(\delta
_{1})K^{T}(t-\theta +\delta _{1})M_{2}(\theta _{2})K(t+\theta _{2}-\xi
+\delta _{2})G(\delta _{2})d\delta _{2}d\delta _{1}d\theta _{2} \\
& +\int_{-h}^{0}\int_{-h}^{0}A_{1}^{T}K^{T}(t+\theta _{1}-\theta
-h)M_{3}(\theta _{1},\theta _{2})K(t+\theta _{2}-\xi -h)A_{1}d\theta
_{1}d\theta _{2} \\
& +\int_{-h}^{0}\int_{-h}^{0}\int_{-h}^{\xi }A_{1}^{T}K^{T}(t+\theta
_{1}-\theta -h)M_{3}(\theta _{1},\theta _{2})K(t+\theta _{2}-\xi +\delta
)G(\delta )d\delta d\theta _{1}d\theta _{2} \\
& +\int_{-h}^{0}\int_{-h}^{0}\int_{-h}^{\theta }G^{T}(\delta )K^{T}(t+\theta
_{1}-\theta +\delta )M_{3}(\theta _{1},\theta _{2})K(t+\theta _{2}-\xi
-h)A_{1}d\delta d\theta _{1}d\theta _{2} \\
& +\int_{-h}^{0}\int_{-h}^{0}\int_{-h}^{\theta }\int_{-h}^{\xi }G^{T}(\delta
_{1})K^{T}(t+\theta _{1}-\theta +\delta _{1})M_{3}(\theta _{1},\theta
_{2})K(t+\theta _{2}-\xi +\delta _{2}) 
 G(\delta _{2})d\delta _{2}d\delta _{1}d\theta _{1}d\theta _{2}%
\bigg)dt.
\end{split}
\label{Pi_2_theta_xi}
\end{equation}

Finally, it shown that \eqref{Pi_2_Trans} and \eqref{Pi_2_theta_xi} are equal.
$\square$
\end{proof}

According to the proof of Proposition \ref{proposicion_2_1}, the Bellman functional \eqref{funcional_v_phi_11} can be expressed in terms of the state $x_t$ as:
\begin{equation}
\begin{split}
V(x_{t}) = & x^{T}(t) \Pi_{0} x(t) +2 x^{T}(t) \int_{-h}^{0}
\Pi_{1}(\theta) x (t+\theta) d \theta\\
&+ \int_{-h}^{0} \int_{-h}^{0} x^{T}(t+\xi) \Pi_{2} (\xi,\theta)
x(t+\theta) d \xi d \theta.
\end{split}
\label{funcional_v_phi_11_x_t}
\end{equation}	

Next, we reveal in our main result the connection between the Bellman functional and the Lyapunov matrix of the closed-loop time-delay systems \eqref{sistema_1_lazo_cerrado_u_gammas}, a fact that is well known in the delay free case. 

\begin{proof}[Proof of Theorem \ref{teorema_2_2}]
According to \eqref{Vpositiva} the expression 
\eqref{funcional_v_phi_11_x_t} is positive
definite. Assuming that $x_{t}$ is a trajectory of system \eqref{sistema1}, and defining
\begin{equation}
H(x_{t},u)=\left. \dot{V}(x_{t})\right\vert _{\substack{\eqref{sistema1} \\ u-admissible}}+x^{T}(t)Qx(t)+u^{T}(t)Ru(t),
\label{H_x_u2}
\end{equation}
where the time derivative of \eqref{funcional_v_phi_11_x_t} along the trajectories of the system \eqref{sistema1} is
\begin{equation}
\begin{split}
&\left. \dot{V}(x_{t})\right\vert _{\substack{\eqref{sistema1}
\\ u-admissible}} =\left( Ax(t)+Bx(t-h)+\int_{-h}^{0}E(\theta )x(t+\theta
)d\theta +Du(t)\right) ^{T}\Pi _{0}x(t) \\
&+x^{T}(t)\Pi _{0}\left( Ax(t)+Bx(t-h)+\int_{-h}^{0}E(\theta )x(t+\theta
)d\theta +Du(t)\right) \\
&+\left( Ax(t)+Bx(t-h)+\int_{-h}^{0}E(\theta )x(t+\theta )d\theta
+Du(t)\right) ^{T}\int_{-h}^{0}\Pi _{1}(\theta )x(t+\theta )d\theta \\
&+x^{T}(t)\int_{-h}^{0}\Pi _{1}(\theta )\frac{\partial }{\partial \theta }
x(t+\theta )d\theta +\int_{-h}^{0}\frac{\partial }{\partial \theta }
x^{T}(t+\theta )\Pi _{1}^{T}(\theta )d\theta x(t) \\
&+\int_{-h}^{0}x^{T}(t+\theta )\Pi _{1}^{T}(\theta )d\theta \left(
Ax(t)+Bx(t-h)+\int_{-h}^{0}E(\theta )x(t+\theta )d\theta +Du(t)\right) \\
&+\int_{-h}^{0}\int_{-h}^{0}\frac{\partial }{\partial \xi }\left(
x^{T}(t+\xi )\right) \Pi _{2}(\xi ,\theta )x(t+\theta )d\theta d\xi   
\\
&+\int_{-h}^{0}\int_{-h}^{0}x^{T}(t+\xi )\Pi _{2}(\xi ,\theta )\frac{\partial }{\partial \theta }\left( x(t+\theta )\right) d\theta d\xi .
\end{split}
\label{V_x_t_punto}
\end{equation}

Substituting \eqref{V_x_t_punto} into \eqref{H_x_u2}, implies that
\begin{equation}
\begin{split}
&H(x_{t},u) =x^{T}(t)\left( A^{T}\Pi _{0}+x^{T}(t)\Pi _{0}A+Q\right)
x(t)\\
&+2x^{T}(t-h)B^{T}\Pi _{0}x(t)+2x^{T}(t)\int_{-h}^{0}A^{T}\Pi _{1}(\theta)x(t+\theta )d\theta\\
&+2x^{T}(t-h)\int_{-h}^{0}B^{T}\Pi _{1}(\theta )x(t+\theta )d\theta +2x^{T}(t)\int_{-h}^{0}\Pi _{1}(\theta )\frac{\partial }{\partial \theta }%
x(t+\theta )d\theta\\
&+\int_{-h}^{0}\int_{-h}^{0}\frac{\partial }{\partial \xi }\left(
x^{T}(t+\xi )\right) \Pi _{2}(\xi ,\theta )x(t+\theta )d\theta d\xi\\
&+\int_{-h}^{0}\int_{-h}^{0}x^{T}(t+\xi )\Pi _{2}(\xi ,\theta )\frac{%
\partial }{\partial \theta }\left( x(t+\theta )\right) d\theta d\xi\\
&+2\int_{-h}^{0}x^{T}(t+\theta )E^{T}(\theta )\Pi _{0}d\theta x(t)+2\int_{-h}^{0}\int_{-h}^{0}x^{T}(t+\xi )E^{T}(\xi )\Pi _{1}(\theta
)x(t+\theta )d\xi d\theta\\
&+2u^{T}(t)\int_{-h}^{0}D^{T}\Pi _{1}(\theta )x(t+\theta )d\theta
+2x^{T}(t)\Pi _{0}Du(t)+u^{T}(t)Ru(t). 
\end{split}
\label{H_x_t_u_1}
\end{equation}

By the fundamental theorem of calculus of variations \cite{kirk2004optimal}
\begin{equation*}
\min_{u-admissible}H(x_{t},u)=H(x_{t}^{\ast },u^{\ast }),
\end{equation*}%
\begin{equation*}
\frac{\partial }{\partial u}H(x_{t},u)=2D^{T}\Pi
_{0}x(t)+2D^{T}\int_{-h}^{0}\Pi _{1}(\theta )x(t+\theta )d\theta +2Ru(t)=0,
\end{equation*}%
when evaluated at  $u=u^*$; we have that%
\begin{equation}
u^{\ast }(t)=-R^{-1}D^{T}\Pi _{0}x(t)-R^{-1}D^{T}\int_{-h}^{0}\Pi
_{1}(\theta )x(t+\theta )d\theta .  \label{u_0_t_optima1}
\end{equation}
Moreover, as%
\begin{equation*}
\frac{\partial ^{2}}{\partial u^{2}}H(x_{t},u)=2R>0,
\end{equation*}%
we conclude that $u^{*}(t)$ is a local minimum of \eqref{H_x_u2}.

For $u^{*}(t)$ in equation \eqref{u_0_t_optima1} to stabilize system \eqref{sistema1} and to be a global minimum, it has to satisfy the
Hamilton-Jacobi-Bellman equation \eqref{inciso_a}.
If $t=0$, $x_{0}=\varphi$, then
\begin{equation}
\left. \dot{V}(\varphi )\right\vert _{\substack{ \eqref{sistema1} \\ u=u^{\ast }}}+\varphi ^{T}(0)Q\varphi (0)+u^{\ast
^{T}}(t)Ru^{\ast }(t)=0,  \label{ecuacion_HJB}
\end{equation}
where%
\begin{equation}
\begin{split}
&\left. \dot{V}(\varphi )\right\vert _{\substack{ \eqref{sistema1} \\ u=u^{\ast }}} =\varphi ^{T}(0)\left( A^{T}\Pi
_{0}+\Pi _{0}A\right) \varphi (0)+2\varphi ^{T}(0)\Pi _{0}B\varphi (-h) 
 \\
&+2\varphi ^{T}(0)\int_{-h}^{0}\Pi _{1}(\theta )\left( \frac{d}{d\theta }%
\varphi (\theta )\right) d\theta+2\varphi ^{T}(0)\int_{-h}^{0}A^{T}\Pi _{1}(\theta )\varphi (\theta
)d\theta    \\
&+2\varphi ^{T}(-h)\int_{-h}^{0}B^{T}\Pi _{1}(\theta )\varphi (\theta
)d\theta+\int_{-h}^{0}\int_{-h}^{0}\left( \frac{d}{d\xi }\varphi ^{T}(\xi )\right)
\Pi _{2}(\xi ,\theta )\varphi (\theta )d\theta d\xi    \\
&+\int_{-h}^{0}\int_{-h}^{0}\varphi ^{T}(\xi )\Pi _{2}(\xi ,\theta )\left( 
\frac{d}{d\theta }\varphi (\theta )\right) d\theta d\xi+2\varphi ^{T}(0)\int_{-h}^{0}\Pi _{0}E(\theta )\varphi (\theta )d\theta  
\\
&+2\int_{-h}^{0}\int_{-h}^{0}\varphi ^{T}(\xi )E^{T}(\xi )\Pi _{1}(\theta
)\varphi (\theta )d\xi d\theta+2u^{\ast T}(t)\int_{-h}^{0}D^{T}\Pi _{1}(\theta )\varphi (\theta )d\theta +2\varphi ^{T}(0)\Pi _{0}Du^{\ast }(t). 
\end{split}
\label{V_phi_punto}
\end{equation}

Substituting \eqref{V_phi_punto} and \eqref{u_0_t_optima1} into the left-hand of \eqref{ecuacion_HJB}, gives
\begin{equation*}
\begin{split}
&\left. \dot{V}(\varphi )\right\vert _{\substack{\eqref{sistema1} \\ u=u^{\ast }}}+\varphi ^{T}(0)Q\varphi (0)+u^{\ast
^{T}}(t)Ru^{\ast }(t) \\
&=\varphi ^{T}(0)\left( A^{T}\Pi _{0}+\Pi _{0}A-\Pi _{0}DR^{-1}D^{T}\Pi
_{0}+\Pi _{1}^{T}(0)+\Pi _{1}(0)+Q\right) \varphi (0) \\
&+2\varphi ^{T}\left( 0\right) \left( \Pi _{0}B-\Pi _{1}(-h)\right) \varphi
(-h) \\
&+2\varphi ^{T}(0)\int_{-h}^{0}\left( \Pi _{0}E(\theta )-\Pi
_{0}DR^{-1}D^{T}\Pi _{1}(\theta )+A^{T}\Pi _{1}(\theta )-\frac{d\Pi
_{1}(\theta )}{d\theta }+\Pi _{2}(0,\theta )\right) \varphi (\theta )d\theta 
\\
&+2\varphi ^{T}(-h)\int_{-h}^{0}\left( B^{T}\Pi _{1}(\theta )-\Pi
_{2}(-h,\theta )\right) \varphi (\theta )d\theta  \\
&+\int_{-h}^{0}\int_{-h}^{0}\varphi ^{T}(\xi )\left( -\frac{\partial \Pi
_{2}(\xi ,\theta )}{\partial \xi }-\frac{\partial \Pi _{2}(\xi ,\theta )}{%
\partial \theta }-\Pi _{1}^{T}(\xi )DR^{-1}D^{T}\Pi _{1}(\theta )+2E^{T}(\xi
)\Pi _{1}(\theta )\right) \varphi (\theta )d\xi d\theta .
\end{split}
\end{equation*}

Therefore, equation \eqref{ecuacion_HJB} is satisfied if and only if the set of equations \eqref{SetPartialEquations} is true. $\square$
\end{proof}

\begin{remark}
\label{remark_1}
The existence of a solution for the set of equations given by \eqref{SetPartialEquations} can be proved by the arguments given in \cite{Ross_1969}. If the conditions of Theorem \ref{teorema_2_2} are satisfied, matrices $\Pi_{0}, \Pi_{1}(\cdot)$ and $\Pi_2(\cdot,\cdot)$ can be related to a function $Z(\xi,\theta)$ as:
\begin{equation}
\begin{split}
\Pi_0 =Z(0,0), \\
\Pi_1(\theta) =-\frac{\partial Z(0,\theta)}{\partial \theta}, \\
\Pi_2(\xi,\theta) =\frac{\partial^{2} Z(\xi,\theta)}{\partial \xi \partial \theta},
\end{split}
\label{zetas}
\end{equation}	
with $Z(\xi,\theta)=Z^{T}(\theta,\xi)$. Now, define the function $W(\xi,\theta)$ as follows:
\begin{equation*}
\begin{split}
W(\xi,\theta)=\frac{\partial Z(\xi,\theta)}{\partial \xi}.
\end{split}
\label{W}
\end{equation*}	
It follows from \eqref{zetas}, and the fact that $Z(\xi,\theta)=Z^{T}(\theta,\xi)$ that
\begin{equation}
\begin{split}
\Pi_2(\xi,\theta) =\frac{\partial W(\xi,\theta)}{\partial \theta},\\
\Pi_1(\theta) =-W^{T}(\theta,0),   \Pi_1^{T}(\theta) =-W(\theta,0).
\end{split}
\label{PIS}
\end{equation}	
By using the third and fourth equations of the set given by \eqref{SetPartialEquations} and the equations \eqref{PIS}, we arrive at the following single second order partial differential equation for $W(\xi,\theta)$:
\begin{equation}
\begin{split}
\frac{\partial^{2} W(\xi,\theta)}{\partial\xi \partial \theta} + \frac{\partial^{2} W(\xi,\theta)}{\partial \theta^{2}}+\left[W(\xi,0)DR^{-1}D^{T}+2E^{T}(\xi)\right]W^{T}(\theta,0)=0.
\end{split}
\label{EDPSO}
\end{equation}	
with a boundary constraint:
\begin{equation}
\begin{split}
\frac{\partial W(-h,\theta)}{\partial \theta} =-B^{T}W^{T}(\theta,0).
\end{split}
\label{CF}
\end{equation}	
The equations given by \eqref{EDPSO} and \eqref{CF}, have the same structure \textbf{as} the equations presented in \cite{Ross_1969}; the conditions for the existence and uniqueness of the solutions of the same structure of equation given by \eqref{EDPSO} were presented in \cite{bernstein2016existence}, \cite{hughes1977well}, \cite{Ross_1969}.
\end{remark}
\textbf{The} following proposition gives the relation between the matrices $\Pi_{0}$, $\Pi_{1}(\cdot)$ and $\Pi_{2}(\cdot,\cdot)$ with the Lyapunov matrix for time delay systems.
\begin{proposition}
 \label{escrituramatlyap}
Matrices \eqref{pi0barra}-\eqref{pi2barra} can be expressed as:
\begin{equation}
\label{Pi_0_con_matrices_Lyapunov}
\Pi_{0} = U(0,M_{1}) + 2 \int_{-h}^{0} U(\theta, M_{2}(\theta)) d\theta 
+\int_{-h}^{0} \int_{-h}^{0} U(-\theta_{1} + \theta_{2}, M_{3}(\theta_1,\theta_{2})) d\theta_{1} d \theta_{2},
\end{equation}
\begin{equation}
    \begin{split}
        \Pi_{1}(\theta ) =&U(-\theta -h,M_{1})A_{1}+\int_{-h}^{\theta }U(-\theta +\delta
		,M_{1})G(\delta )d\delta
		+\int_{-h}^{0}U(\theta _{2}-\theta -h,M_{2}(\theta
		_{2}))d\theta _{2}A_{1}\\
		&+\int_{-h}^{0}U(-\theta _{2}-\theta -h,M^{T}_{2}(\theta
		_{2}))d\theta _{2}A_{1}
		+\int_{-h}^{0}\int_{-h}^{\theta }U(\theta _{2}-\theta +\delta
		,M_{2}(\theta _{2}))G(\delta )d\delta d\theta
		_{2}\\
		& +\int_{-h}^{0}\int_{-h}^{\theta }U(-\theta _{2}-\theta +\delta
		,M^{T}_{2}(\theta _{2}))G(\delta )d\delta d\theta
		_{2}
		+\int_{-h}^{0}\int_{-h}^{0}U(-\theta _{1}+\theta _{2}-\theta
		-h,M_{3}(\theta _{1},\theta _{2}))d\theta _{1}d\theta _{2}A_{1}\\ 
		& +\int_{-h}^{0}\int_{-h}^{0}\int_{-h}^{\theta }U(-\theta _{1}+\theta_{2}-\theta +\delta ,M_{3}(\theta _{1},\theta _{2}))G(\delta )d\delta d\theta_{1}d\theta _{2},
	\end{split}
		\label{ec_55_pi_1}
\end{equation}
	\begin{equation}
	\begin{split}
		\Pi_{2}(\xi ,\theta ) =&A_{1}^{T}U(\xi -\theta ,M_{1})A_{1}
		+A_{1}^{T}\int_{-h}^{\theta}U(\xi +h-\theta +\delta ,M_{1})G(\delta )d\delta
		+\int_{-h}^{\xi }G^{T}(\delta )U(\xi -\delta -\theta
		-h,M_{1})A_{1}d\delta\\
		&+\int_{-h}^{\xi }\int_{-h}^{\theta }G^{T}(\delta
		_{1})U(\xi -\theta -\delta _{1}+\delta _{2},M_{1})G(\delta _{2})d\delta_{2}d\delta _{1}
		+2A_{1}^{T}\int_{-h}^{0}U(\xi +\theta _{2}-\theta ,M_{2}(\theta
		_{2}))d\theta _{2}A_{1}\\
		&+2A_{1}^{T}\int_{-h}^{0}\int_{-h}^{\theta }U(\xi +h+\theta
		_{2}-\theta +\delta ,M_{2}(\theta _{2}))G(\delta )d\delta d\theta
		_{2}\\
		 &+2\int_{-h}^{0}\int_{-h}^{\xi }G^{T}(\delta )U(\xi -\delta +\theta_{2}-\theta -h,M_{2}(\theta _{2}))d\delta d\theta _{2}A_{1}\\
		 &+2\int_{-h}^{0}\int_{-h}^{\xi }\int_{-h}^{\theta}G^{T}(\delta
		_{1})U(\xi -\delta _{1}+\theta _{2}-\theta +\delta _{2},M_{2}(\theta
		_{2}))
		G(\delta _{2})d\delta _{2}d\delta _{1}d\theta _{2}\\
		&+A_{1}^{T}\int_{-h}^{0}\int_{-h}^{0}U(-\theta _{1}+\xi +\theta
		_{2}-\theta ,M_{3}(\theta _{1},\theta _{2}))d\theta _{1}d\theta _{2}A_{1}\\
		 & +A_{1}^{T}\int_{-h}^{0}\int_{-h}^{0}\int_{-h}^{\theta }U(-\theta
		_{1}+\xi +h+\theta _{2}-\theta +\delta ,M_{3}(\theta _{1},\theta
		_{2}))
		G(\delta )d\delta d\theta _{1}d\theta _{2}\\
		& +\int_{-h}^{0}\int_{-h}^{0}\int_{-h}^{\xi }G^{T}(\delta )U(-\theta
		_{1}+\xi -\delta +\theta _{2}-\theta -h,M_{3}(\theta _{1},\theta_{2}))
		 d\delta d\theta _{1}d\theta _{2}A_{1}\\
	     & +\int_{-h}^{0}\int_{-h}^{0}\int_{-h}^{\xi }\int_{-h}^{\theta
		}G^{T}(\delta _{1})U(-\theta _{1}+\xi -\delta _{1}+\theta _{2}-\theta
		+\delta _{2},M_{3}(\theta _{1},\theta _{2}))
		G(\delta _{2})d\delta
		_{2}d\delta _{1}d\theta _{1}d\theta _{2}.
	    \label{ec_pi_2_explicita}    
	    \end{split}
	\end{equation}
\end{proposition}

\begin{proof}[Proof of Proposition \ref{escrituramatlyap}]
The expression of \eqref{pi0barra}-\eqref{pi2barra} in terms of Lyapunov matrices is revealed by introducing in each summand of the matrices, the solution \eqref{sol_cauchy_reducida}, and, when needed, equation \eqref{k_gorro}. The result is obtained by using the Definition \ref{def_matris_lyapunov} for the delay Lyapunov matrix function and its properties given in Theorem \ref{teorema_santos}. The three terms of  equation \eqref{pi0barra} corresponding to $\Pi_0$ can be expressed straightforwardly as:
\begin{equation}
\int_{0}^{\infty} K^{T}(t) M_{1} K(t) dt=U(0,M_{1}),
\label{u_0_m1_de_pi_0}
\end{equation}
\begin{equation}
2 \int_{0}^{\infty} \int_{-h}^{0}  K^{T}(t)M_{2}(\theta)K(t+\theta) %
 d\theta dt= 2 \int_{-h}^{0} U(\theta, M_{2}(\theta)) d\theta,
\label{int_2_de_pi_0}
\end{equation}
and, by some change of variable,
\begin{equation}
\begin{split}
&\int_{0}^{\infty} \int_{-h}^{0} \int_{-h}^{0} K^{T}(t+%
\theta_{1})M_{3}(\theta_{1},\theta_{2})K(t+\theta_{2})  d \theta_{1}
d \theta_{2} dt
= \int_{-h}^{0} \int_{-h}^{0} U(-\theta_{1} + \theta_{2}, M_{3}(\theta_1 ,
\theta_{2})) d\theta_{1} d \theta_{2}.
\end{split}
\label{int_3_de_pi_0}
\end{equation}

Now, substituting the right-hand side of \eqref{u_0_m1_de_pi_0}-\eqref{int_3_de_pi_0}  into \eqref{pi0barra} we arrive at (\ref{Pi_0_con_matrices_Lyapunov}).
The same procedure is applied to the summands of the expressions \eqref{pi1barra} for $\Pi_{1}(\cdot)$ and \eqref{pi2barra} for  $\Pi_{2}(\cdot,\cdot)$ to obtain (\ref{ec_55_pi_1}). The same procedure is applied to the summands of the expressions \eqref{pi1barra} for $\Pi_{1}(\cdot)$ and \eqref{pi2barra} for  $\Pi_{2}(\cdot,\cdot)$ to obtain (\ref{ec_55_pi_1}) and (\ref{ec_pi_2_explicita}). $\square$ 
\end{proof}

The above result has a consequence of major importance:

\begin{corollary} \label{unicidad_pi_s}
The matrix $\Pi_{0}$ \textit{and the matrix functions} $\Pi_{1}(\cdot)$, $\Pi_{2}(\cdot,\cdot)$ \textit{in Proposition \ref{escrituramatlyap} exist and are unique}.
\end{corollary}

\begin{proof}[Proof of Corollary \ref{unicidad_pi_s}]
The result follows immediately from the expressions of these matrices in Proposition \ref{escrituramatlyap} and from Theorem \ref{exisuni} which establishes the  existence and uniqueness of the delay Lyapunov matrix of stable systems. $\square$
\end{proof}

Having fully established the form of the Bellman functional, it is now possible to follow the steps of Ross, namely to compute the explicit time derivative of the Bellman functional along the trajectories of system \eqref{sistema1}, to replace it into the Bellman equation \eqref{inciso_a} to find  necessary optimality conditions. The structure of the optimal control and the five equations 1)-5) in Theorem \ref{teorema_2_2} are obtained. According to Theorem \ref{teorema_2_1} given in \cite{Ross_1969}, these conditions are sufficient as well.

\begin{remark}
In our previous work a sub-optimal control for time-delay systems was presented \cite{Santos_2009}. There, an approximation of the Bellman functional was obtained via the prescribed derivative functional approach and the delay Lyapunov matrix definition. A surprising fact was that the thirteen summands obtained approximation seemed much more complex than the three terms Bellman functional \eqref{funcional_v_phi_11_x_t}. Proposition \ref{escrituramatlyap} allows to establish via Fubini's theorem \cite{Thomas}, that both functionals have in fact the same form \cite{ortega2020}.
 It is worthy of mention that this equivalence was validated by  numerical verification on an example presented in \cite{Ross71}, \cite{Santos_2009}.
\end{remark}

\section{Lower and upper bounds for the Bellman functional}

We establish in this section that the Bellman functional admits a quadratic local upper bound and a cubic local lower bound. Our proof is inspired in \cite{Wenzhang_1989} where such bounds are given for Lyapunov-Krasovskii functionals with prescribed derivative in the case of pointwise and distributed delay linear systems.

\begin{proposition} \label{prop_cotas}
	Given the stable system \eqref{sistema_1_lazo_cerrado_u_gammas}, the Bellman functional \eqref{funcional_v_phi_11} admits an upper quadratic bound and a lower local cubic bound.
\end{proposition}

\begin{proof}[Proof of Proposition \ref{prop_cotas}]
 Given the functional \eqref{funcional_v_phi_11}, using the fact that $||\varphi (\theta )||\leq ||\varphi ||_{h}$, and appropriate majorizations, we get \begin{equation*}
\begin{split}
&V(\varphi )\leq ||\varphi||_{h}^{2}||\Pi _{0}||+2||\varphi
||^{2}_{h}\int_{-h}^{0}||\Pi _{1}(\theta )||d\theta
+||\varphi||_{h}^{2}\int_{-h}^{0}\int_{-h}^{0}||\Pi _{2}(\xi ,\theta )||d\xi d\theta,    
\end{split}
\end{equation*}
hence,
\begin{equation}
 V(\varphi )\leq C_{1}||\varphi ||_{h}^{2},      \ C_{1}=||\Pi _{0}||+2X_{1}+X_{2}>0,
 \label{C1}
\end{equation} with
\begin{eqnarray*}
	X_{1} =\int_{-h}^{0}||\Pi _{1}(\theta )||d\theta,  X_{2} =\int_{-h}^{0}\int_{-h}^{0}||\Pi _{2}(\xi ,\theta )||d\xi d\theta,
\end{eqnarray*}
and we conclude that the Bellman functional admits an upper quadratic bound. Now, for the lower bound for functional \eqref{funcional_v_phi_11}, consider the stable closed-loop system \eqref{sistema_1_lazo_cerrado_u_gammas}, and notice that functional \eqref{funcional_v_phi_11} satisfies:
\begin{equation}
\frac{dV(x_{t}^{\ast })}{dt}=-x^{\ast T}(t)Qx^{\ast }(t)-u^{\ast T}(t)Ru^{\ast }(t)<0.
\label{derivada_V_x_t_estrella}
\end{equation}

Here $x_{t}^{\ast }$ denotes the optimal trajectory of the closed-loop system with optimal control $u^{\ast}$. Following the ideas of the proof of the main Theorem given in \cite{Wenzhang_1989} and \textbf{\cite{medvedeva2015synthesis}}.
Integrating the closed loop system given by \eqref{sistema_1_lazo_cerrado_u_gammas}, from zero to $t$, we get
\begin{equation*}
x^{\ast }(t)-x^{\ast }(0)=A_0\int_{0}^{t}x^{\ast
}(s)ds+A_1\int_{0}^{t}x^{\ast }(s-h)ds+\int_{0}^{t}\int_{-h}^{0}G(\theta
)x^{\ast }(s+\theta )d\theta ds
\end{equation*}%
hence, for $g=\underset{\theta \in \lbrack -h,0]}{sup}\left\Vert G(\theta
)\right\Vert$, some variable changes, and integral properties, we obtain that:
\begin{eqnarray*}
\left\Vert x^{\ast }(t)\right\Vert  &\leq &\left\Vert \varphi (0)\right\Vert
+\left\Vert A_0\right\Vert \int_{0}^{t}\left\Vert x^{\ast }(s)\right\Vert
ds+\left\Vert A_1\right\Vert \int_{0}^{t}\left\Vert x^{\ast
}(s-h)\right\Vert ds
+g\int_{0}^{t}\int_{-h}^{0}\left\Vert x^{\ast }(s+\theta)\right\Vert d\theta ds  \\
&\leq &(\left\Vert \varphi (0)\right\Vert +(\left\Vert A_1\right\Vert
+gh)\int_{-h}^{0}\left\Vert \varphi (\theta )\right\Vert d\theta
)+(\left\Vert A_0\right\Vert +\left\Vert A_1\right\Vert
+gh)\int_{0}^{t}\left\Vert x^{\ast }(\theta )\right\Vert d\theta
\end{eqnarray*}%
and, by Bellman-Gronwall Lemma \cite{bellman1963differential}, we get
\begin{eqnarray*}
\left\Vert x^{\ast }(t)\right\Vert  &\leq &m_{0}e^{Lt} \\
\text{with }m_{0} &=&\left\Vert \varphi (0)\right\Vert +(\left\Vert
A_1\right\Vert +gh)\int_{-h}^{0}\left\Vert \varphi (\theta )\right\Vert
d\theta  \\
\text{and }L &=&\left\Vert A_0\right\Vert +\left\Vert A_1\right\Vert +gh.
\end{eqnarray*}%
For $\left\Vert \varphi \right\Vert _{h}\leq \alpha ,$we get%
\begin{equation}
\left\Vert x(t^{\ast},\varphi )\right\Vert \leq N(t),\quad N(t)=\alpha
(1+\left\Vert A_1\right\Vert h+gh^{2})e^{Lt},\text{ for }t\geq 0,
\label{ineq1}
\end{equation}%
with $m_{0}\leq \alpha \left( 1+(\left\Vert A_1\right\Vert +gh)h\right) .$
By the inequality given by (\ref{ineq1}), $\left\Vert \varphi \right\Vert
_{h}\leq \alpha <N(t).$ Consequently $\left\Vert x(t,\varphi )\right\Vert
\leq N(t),$ $t\geq -h$.

Therefore, for all $t\in \lbrack 0,t^{\ast }]$ and $\varphi \in C_{\alpha}=\left\{\varphi : \varphi \in PC([-h,0],R^{n}),	\|\varphi\|\leq\alpha\right\}$, we have:
\begin{equation}
\|x^{\ast}(\varphi,t)\| \leq LN(t)\leq LN(t^{\ast }), \quad \text{for} \quad \varphi \in C_{\alpha }, t \in [0,t^{\ast}],
\label{desigualdadCero}
\end{equation}
notice that $N(t^{\ast})>\alpha $.

Taking norms on both sides of   \eqref{sistema_1_lazo_cerrado_u_gammas} and using standards majorizations and inequalities yields
\begin{equation}
||\dot{x}^{\ast }(t)||\leq C_{2}||x_{t}^{\ast }||_{h},\forall t\in
\lbrack 0,\infty \lbrack,
\label{desigualdad_x_punto}
\end{equation}
with \begin{equation}
 C_{2}= ||A_0||+||A_1||+\int_{-h}^{0}||G(\theta )||d\theta >0. 
 \label{C2}
\end{equation}

Now, the
inequality (\ref{desigualdadCero})
implies that $\left\Vert x^{\ast }(t,\varphi )\right\Vert\leq\left\Vert x^{\ast }(t,\varphi )\right\Vert _{h}\leq LN(t^{\ast
})$, ($\varphi (\theta )=0,$ for all $\theta <-h,$ and $\left\Vert \varphi
\right\Vert \leq \alpha <N(t)$) for $t\in \left[ -h,t^{\ast }\right] $. 
By the following inequality  \cite{aman2008analysis}:
\begin{equation*}
\left\Vert \int_{0}^{t}x^{*\prime }(\tau )d\tau \right\Vert \leq
\int_{0}^{t}\left\Vert x^{*\prime }(\tau )\right\Vert d\tau ,
\end{equation*}%
then we have that
\begin{equation*}
\left\Vert x(0)-x^{*}(t)\right\Vert =\left\Vert x^{*}(t)-x(0)\right\Vert \leq
\int_{0}^{t}\left\Vert x^{*\prime }(\tau )\right\Vert d\tau .
\end{equation*}%
Consider the inequality $\left\Vert H\right\Vert -\left\Vert I\right\Vert
\leq \left\Vert H-I\right\Vert $ follows that%
\begin{equation*}
\left\Vert x(0)\right\Vert -\left\Vert x(t)\right\Vert \leq
\int_{0}^{t}\left\Vert x^{\prime }(\tau )\right\Vert d\tau .
\end{equation*}%
Thus, integrating both sides of $\left\Vert \dot{x^{*}}(t,\varphi )\right\Vert
\leq C_2LN(t^{\ast })$ respect to $t,$ from $0$ to $t,$ we can arrive to
\begin{equation*}
\left\Vert \varphi (0)\right\Vert -\left\Vert x^{*}(t,\varphi )\right\Vert \leq
\int_{0}^{t}\left\Vert x^{*\prime }(\tau ,\varphi )\right\Vert d\tau \leq
C_2LN(t^{\ast })\int_{0}^{t}d\tau
\end{equation*}%
and multiplying by $-1$ both sides:
\begin{equation*}
\left\Vert x(t^{*},\varphi )\right\Vert -\left\Vert \varphi (0)\right\Vert \geq
-C_2LN(t^{\ast })t, \text{ }  t\in \lbrack -h,t^{\ast }].
\end{equation*}%
Hence, we can obtain:
\begin{equation}
||x^{*}(\varphi ,t)||\geq ||\varphi (0)||-\bar{N}t, \qquad \forall t\in \left[
0,t^{\ast }\right],
\label{desigualdad_4}
\end{equation}
 where
\begin{equation}
\bar{N}=\max \left\{ C_{2}LN(t^{\ast }),\frac{\alpha }{2t^{\ast }}%
\right\}.
\label{Nbar}
\end{equation}%
For 
\begin{equation}
\delta =\frac{||\varphi (0)||}{2\bar{N}} \leq t^{\ast} 
\label{delta}
\end {equation}
it follows \eqref{desigualdad_4} that
\begin{equation}
||x^{*}(\varphi ,t)||\geq ||\varphi (0)||-\bar{N}\delta. \label{desigualdad_5}
\end{equation}
therefore $\forall t\in \lbrack 0,\delta]$,
\begin{equation}
 ||x^{*}(\varphi ,t)||\geq ||\varphi (0)||-\bar{N}\frac{||\varphi (0)||}{2\bar{N}}%
= \frac{||\varphi (0)||}{2}>0.\\
\label{desigualdad_x_phi_medios}
\end{equation}
We can see from \eqref{derivada_V_x_t_estrella} that
\begin{equation}
\dot{V}(x_{t}^{\ast })|_{(\ref{sistema1})}\leq -x^{\ast T}(t)Qx^{\ast }(t),
\label{desigualdad_9}
\end{equation}
and $\delta <t^{\ast }$. Integrating both sides of (\ref{desigualdad_9}) gives
\begin{equation}
\int_{0}^{\delta }\frac{d}{dt}V(x_{t}^{\ast })dt\leq -\int_{0}^{\delta
}x^{\ast T}(t)Qx^{\ast }(t)dt.  \label{desigualdad_10}
\end{equation}

Solving the integral on the right-hand side, and using Rayleigh inequality together with (\ref{desigualdad_x_phi_medios}, \ref{delta}) on the left-hand side  gives
\begin{equation*}
V(x_{\delta }^{\ast })-V(\varphi (0))\leq -\frac{\lambda_{min}(Q)}{8\bar{N}}%
||\varphi (0)||^{3}=-u_{\alpha }(||\varphi (0)||).
\end{equation*}

As \eqref{Vpositiva} implies that $V(x^{*}_\delta)>0$, we get
\begin{equation*}
u_{\alpha }(||\varphi (0)||)\leq V(\varphi (0)).
\end{equation*}
Observe that $u_{\alpha }(||\varphi (0)||)$ is monotone increasing and $u_{\alpha }(0) = 0$. 
Clearly, this local cubic lower bound depends on the instantaneous state.  $\square$
\end{proof}

\begin{remark}
The above proposition proves that the Bellman functional is positive definite. Additionally, by Theorem \ref{teorema_2_1} in \cite{hale1993}, we conclude that the optimal solution of system \eqref{sistema_1_lazo_cerrado_u_gammas} is asymptotically uniform stable, a result that was not formally proved in \cite{Ross_1969}.
\end{remark}

\section{Numerical example}
In this section, we present the bounds for the Bellman functional for the optimal control introduced in \cite{Ross71}. The performance index (\ref{desempenioJ}) is such that $Q=\text{diag}\{1,10,1,100\}$, $R=\text{diag}\{1,1\}$, the system is  of the form (\ref{sistema1}) with:
\[
A=\left[
\begin{array}{cccc}
-4.93 & -1.01 & 0 & 0 \\
-3.2 & -5.3 & -12.8 & 0 \\
6.4 & 0.347 & -32.5 & -1.04 \\
0 & 0.833 & 11 & -3.96%
\end{array}%
\right]
\]%
\[
B=\left[
\begin{array}{cccc}
1.92 & 0 & 0 & 0 \\
0 & 1.92 & -12.8 & 0 \\
6.4 & 0.347 & -32.5 & -1.04 \\
0 & 0.833 & 11 & -3.96%
\end{array}%
\right], 
D=\left[
\begin{array}{cc}
1 & 0 \\
0 & 1 \\
0 & 0 \\
0 & 0%
\end{array}%
\right],
\]
and the initial condition is: $\varphi^{T}(\theta)=[\begin{array}{cccc}
0.1 & 0 & 0 & 0 \end{array}] $, $\theta \in [-1,0]$. For experimental results on this control law, the reader is referred to \cite{Lopez_2018}.

To determine the lower bound, we first compute the fundamental matrix of the closed-loop system (\ref{sistema_1_lazo_cerrado_u_gammas}), via the command 'dde23' in MATLAB, see  Figure \ref{Kmatrix}.

\begin{figure}[h]
  \centering
  \includegraphics[width=200pt]{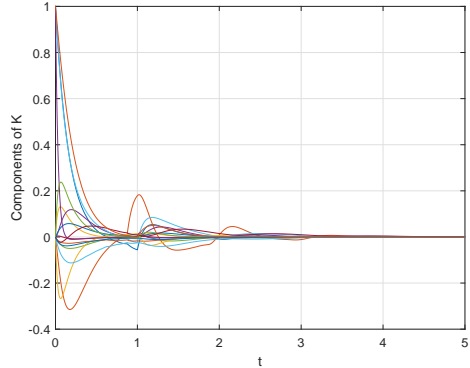}
  \centering
  \caption{Entries of the matrix $K(t)$}
  \label{Kmatrix}
\end{figure}

The value of $t^{\ast }$ is fixed to 1 because all the entries of matrix $%
K(t)$ start to converge to zero after this time. Equation \eqref{C2} gives $C_{2}=40.3438$.
Now, as $\left\Vert \varphi \right\Vert =0.1$, $g=3.0393$, $\left\Vert A_{1}\right\Vert
=1.92$, then $m_0=0.5959$ and $\alpha =0.1$, which satisfies $m_{0}\leq \alpha \left(1+(\left\Vert A_1\right\Vert +gh)h\right)$. With $L=41.9333$, according to \eqref{ineq1}, the function $N(t)$ is:  \begin{equation*}
N(t) =0.5959e^{41.9333t}
\end{equation*}%
Then $N(1)=9.6961\times10^{17}$, it follows from \eqref{Nbar} that $\bar{N}=\max \left\{ 1.6403\times 10^{21};0.05\right\} =1.6403\times 10^{21},$ and that  %
\begin{equation*}
\delta =\frac{\left\Vert \varphi \left( 0\right) \right\Vert }{2\bar{N}}%
=3.0482\times10^{-23},
\end{equation*}%
with $\delta=3.0482\times10^{-23} <1=t^{\ast }.$
As $\lambda_{min}(Q) =1,$
\begin{equation*}
u_{\alpha }\left( \left\Vert x(t)\right\Vert \right)=7.6206\times10^{-23}\left\Vert x(t)\right\Vert
^{3}.
\end{equation*}%

According to \eqref{C1}, the upper bound is%
\begin{equation*}
V(\varphi )\leq 22.574\left\Vert \varphi \right\Vert
_{h}^{2}.
\end{equation*}

\section{Experimental results}
The experimental platform used here is similar to the one used in \cite{Ortega2021}. This dispositive emulates a real atmospheric dehydrator. It has a drying section with a wind tunnel as output and a pipe that recycles the hot air into the system and induces a state delay in the mathematical model. The dehydrator includes of : a temperature sensor $LM35$ with a measurement rate of $10 mV/^{o}C$; a fan
producing a constant air flow with velocity of $2.1 m/s$; an electrical grid (actuator) as heat source; a control voltage in the range $0 - 120 Vrms$ of AC power, which regulates the temperature inside the chamber. This temperature plant has the following linear model (in a specific operation region, see \cite{Lopez_2018}) :
\begin{equation}
\begin{split}
&\dot{x}(t)=a_{0}x(t)+a_{1}x(t-h)+bu(t),\\
\end{split}
\label{sistemaEx}
\end{equation}
where $x(t)$ represents the temperature (process variable) of the drying wind, $\varphi=17^{o}C$ is the initial temperature of the system and the input $u(t)$ is the voltage applied to the grid. 
A least mean square recursive algorithm was used in order to estimate the parameters of the model given by \eqref{sistemaEx}.  The estimated model parameters converge to $a_{0}=-0.046502$, $a_{1}=0.044844$ and $b=0.000143$. The delay $h=4$ seconds is heuristically estimated by comparison of different measurements in the recirculating tube. The estimation error converges to $0.000043153$. Figure \ref{parametros_planta} shows the parameters behaviour along \textbf{with} the identification process.
\begin{figure}[h]
\centering
\subfigure{\includegraphics[scale=0.3]{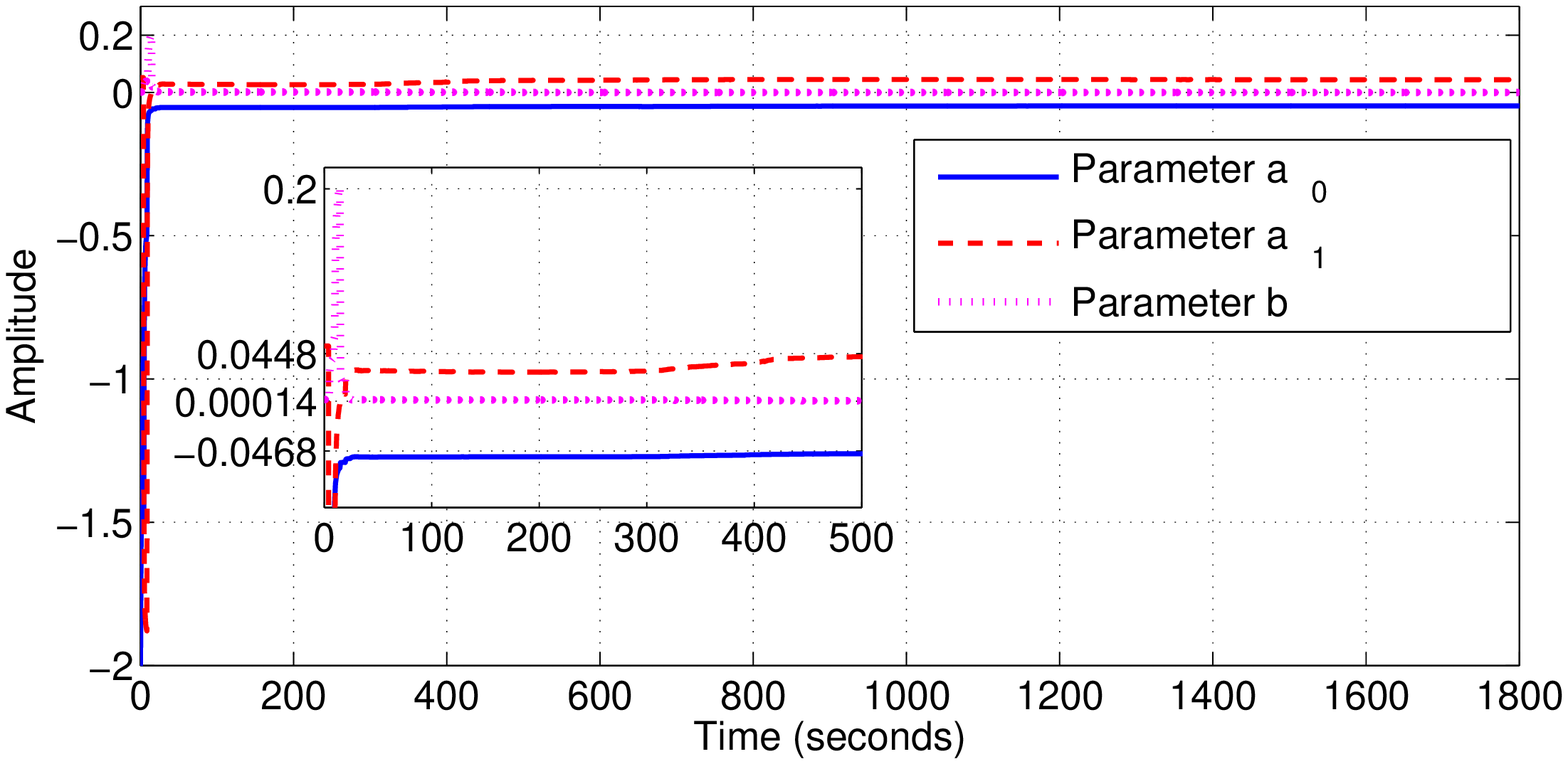}}
\subfigure{\includegraphics[scale=0.3]{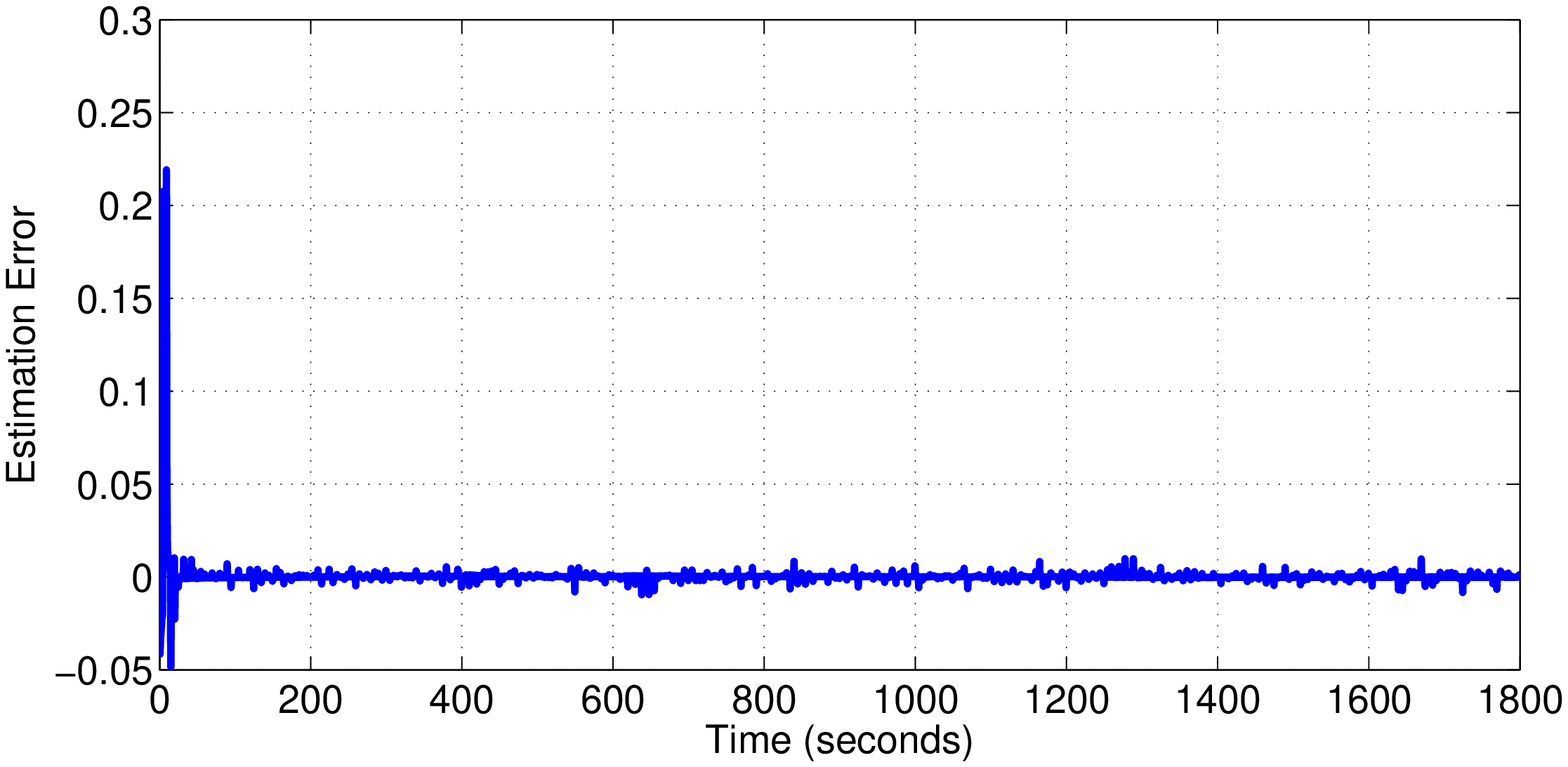}}
\caption{Numerical behaviour of estimated parameters for the linear delay model and estimation error
when a set point fixed to $25^{o}C$ is considered}
\label{parametros_planta}
\end{figure}

In contrast to previous experimental results reported in the specialized literature, see \cite{Lopez_2018}, \cite{rodriguez2019robust} and \cite{Ortega2021}, a trajectory tracking of the process variable is presented here. In fact, let the following piece-wise function
\begin{equation}
	r\left( t\right) =\left\{ 
	\begin{array}{cc}
	\frac{t}{10}+r_{01}, & 0s\leq t<40s, \\ 
	r_{1}, & 40s\leq t<600s, \\ 
	r_{1}-\left( \frac{t-600}{10}\right),  & 600s\leq t<640s, \\ 
	r_{0}, & 640s\leq t<1240s, \\ 
	\left( \frac{t-1240}{10}\right) +r_{02}, & 1240s\leq t<1280s, \\ 
	r_{1}, & 1280s\leq t<1800s,%
	\end{array}%
	\right.
	\label{rampa}
\end{equation}
where $r_{01}$ is the measured initial condition  $17^{o} C$, $r_{02}$ is $18.5^{o} C$ and $r_{1}$, the maximum temperature in the experiment is $25^{o} C$. The aim is the trajectory tracking given by equation \eqref{rampa}. Two control strategies are considered: optimal control given by \eqref{ley_control_u0} and a PI control optimally tuned with the method proposed in \cite{he2000pi}. The optimal control for time delay systems is programmed on a MyRIO-National Instruments target which uses LabVIEW software and sampling time of 500 milliseconds. For the optimal PI control implemented with an industrial Honeywell DC1040 controller, the mathematical model is assumed to be given by
 \begin{equation}
G(s)=\frac{Y(s)}{U(s)}= \frac{\mathcal{K}e^{-\tau s}}{\mathcal{T}s+1},
\label{scalar_modelPI}
\end{equation}%
with parameters $\mathcal{K}=0.01455$, $\mathcal{T}=150$ seconds and $\tau$ = 3 seconds obtained from the step response  (120 VCA rms, applied to the grid, representing 100$\%$ of the control signal). For a quadratic performance index with $Q$=diag$\{$15,15$\}$ and $R=1$, the resulting Optimal PI, has gains $K_{p}=79.51$ and $K_{i}=3.873$, the time delay in the input was compensated by the Dead-Band Time instruction available in the DC 1000 series digital PID Honeywell controllers. The temperature, control and  error signals of both controllers, are depicted in Figure \ref{comparaciones}.

\begin{figure}[h]
\centering
\subfigure{\includegraphics[scale=0.3]{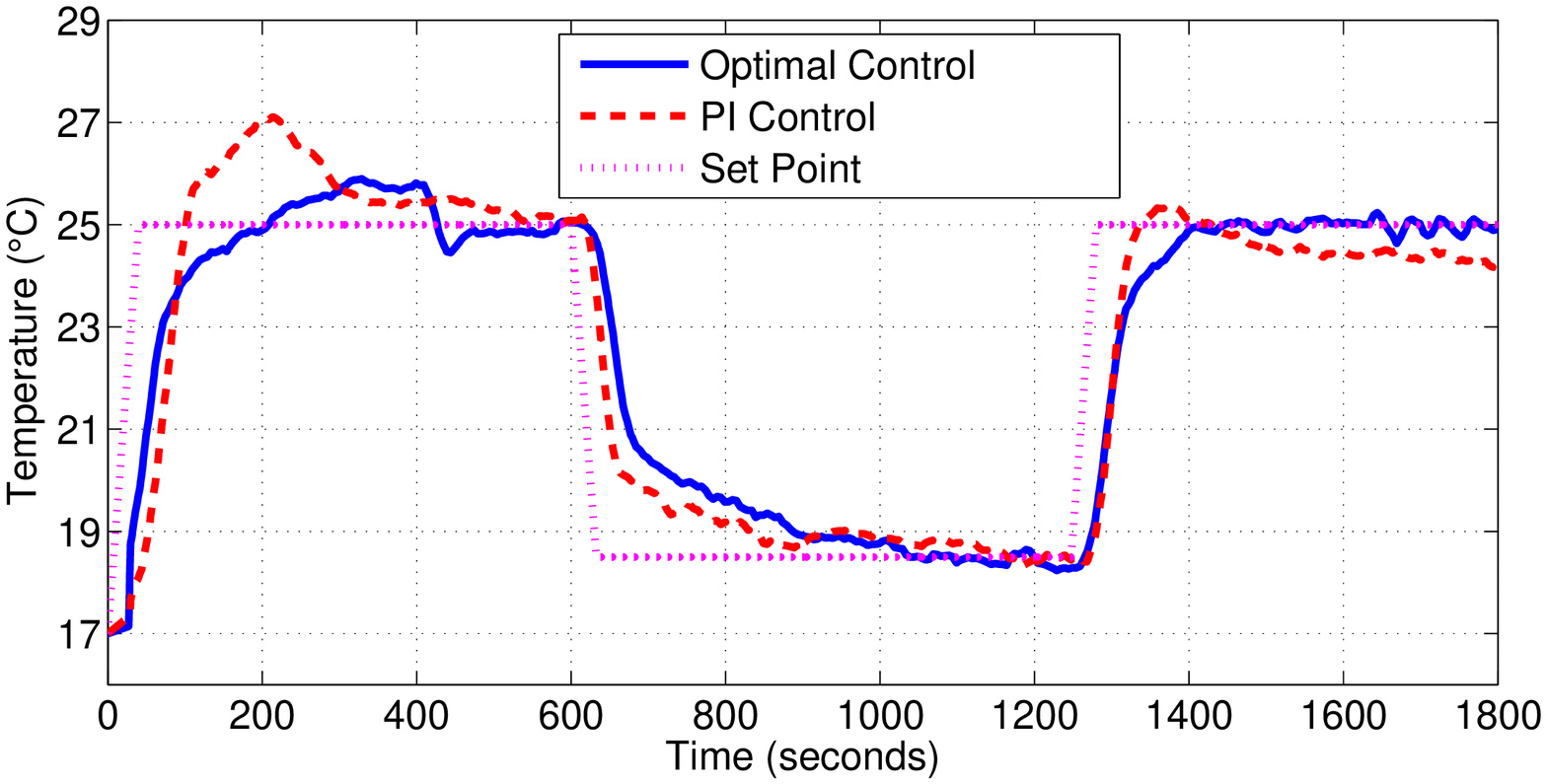}}
\subfigure{\includegraphics[scale=0.3]{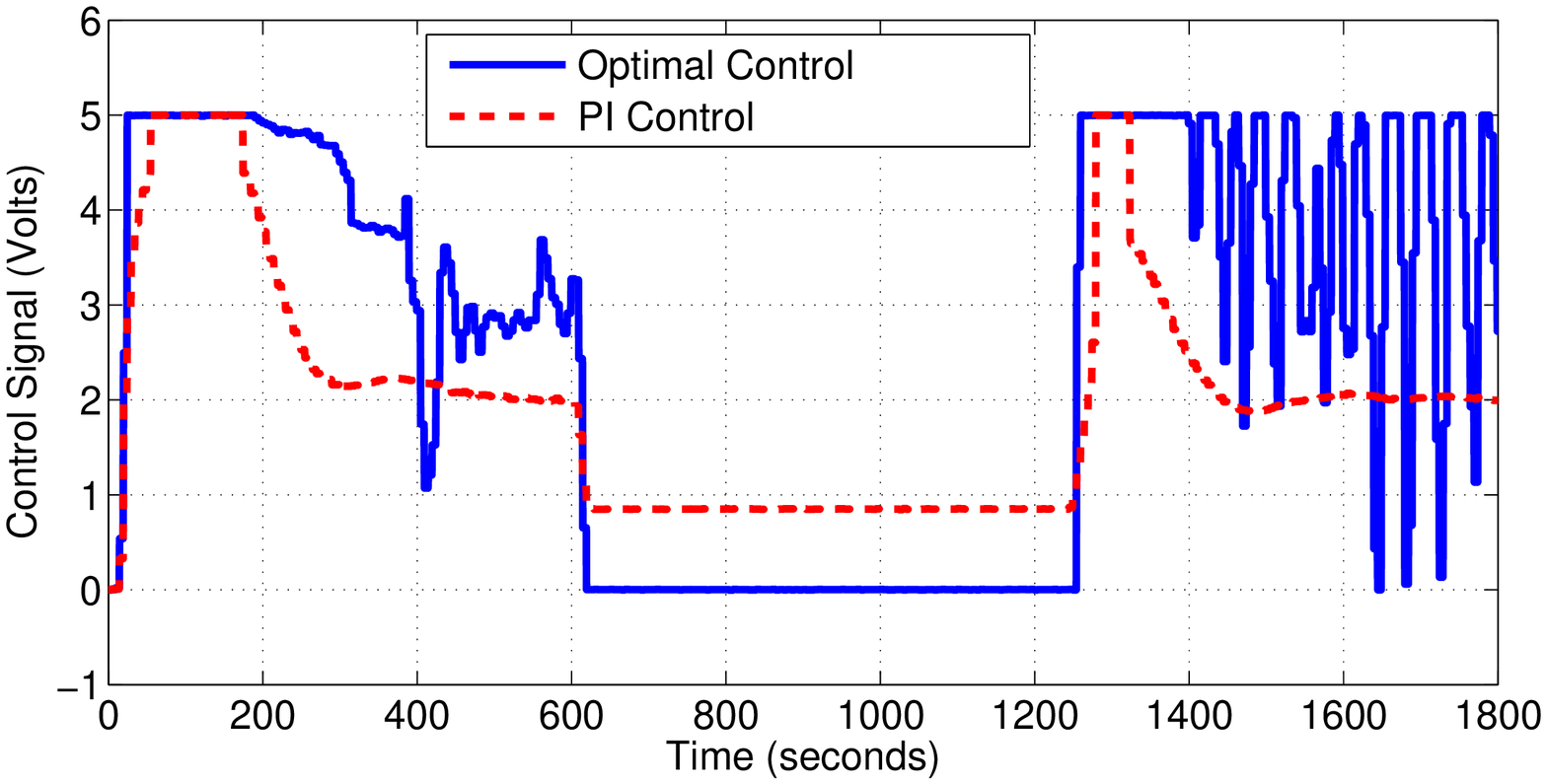}}
\subfigure{\includegraphics[scale=0.3]{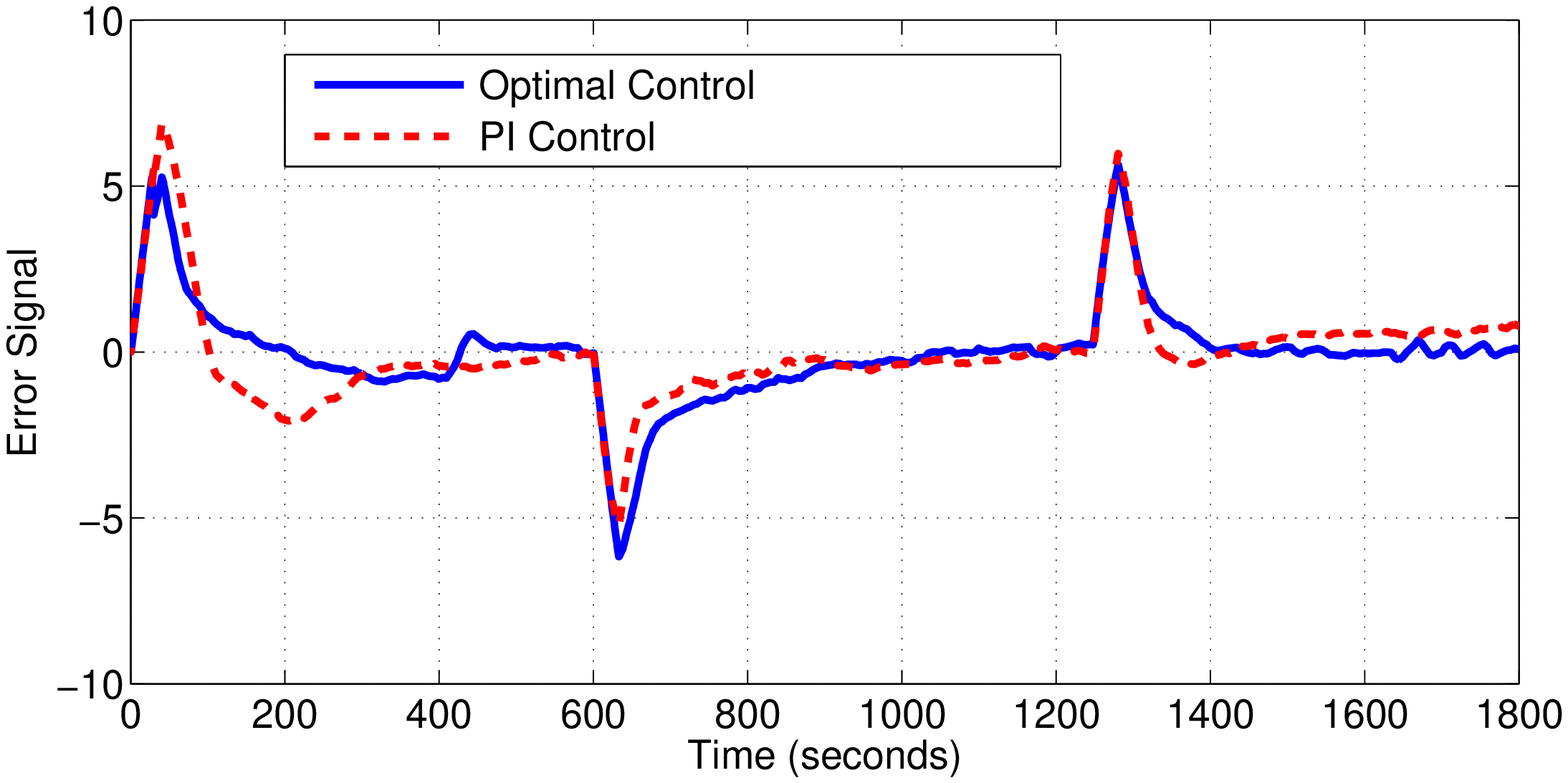}}
\caption{Temperature response, control and error signals with SP (Set Point) adjusted to the trajectory tracking $r(t)$ using both controllers}
\label{comparaciones}
\end{figure}

A better behaviour of the Process Variable without overshot is obtained when the optimal controller given by \eqref{ley_control_u0} is implemented. Indeed, the experimental evidence shown in Table \ref{modelos_continuo} supports this claim.

\begin{table}[h]
\begin{center}
\begin{tabular}{||c|c|c||}
\hline\hline
\textbf{Control strategy} & \textbf{IAE-Optimal Control} &  \textbf{Energy consumption (Wh)} \\ \hline
Optimal control  & 1458.9 & 21.18   \\ \hline
Optimal PI control & 1683.13 & 26.07 \\  \hline\hline
\end{tabular}%
\end{center}
\caption{Comparative table of numerical values for the IAE and energy consumption for the Optimal Control and Optimal PI control}
\label{modelos_continuo}
\end{table}

\section{Conclusions}

The form of the Bellman functional and properties that are the starting point of the optimal control problem of linear distributed time-delay systems are formally justified. Moreover, the solution to the infinite horizon optimal control problem for distributed time delay systems is presented. The expression of the Bellman functional in terms of the delay Lyapunov matrix, allows proving its existence and uniqueness.  We also show that the Bellman functional admits a quadratic upper bound and a local cubic lower bound, which implies that the Bellman functional is positive definite. Some experimental results give evidence of the effectiveness of the optimal control on trajectory tracking tests.

The strategy we employ in this paper will be used in future research to present the Bellman functional for the optimal control problem of other classes of delay systems, in particular those of neutral type. 

\textbf{\textit{Acknowledgements}} 
This work was supported by Conacyt-Mexico Projects: 239371, Conacyt A1-S-24796, SEP-Cinvestav 155.

\bibliographystyle{unsrt}  
\bibliography{references}  

\end{document}